\documentclass[11pt, reqno]{amsart}

\usepackage[dvipsnames,usenames]{color}
\usepackage[colorlinks=true, urlcolor=NavyBlue, linkcolor=NavyBlue, citecolor=NavyBlue]{hyperref}
\usepackage{float}
\usepackage{amsmath, amsthm, amssymb}
\usepackage{amsfonts}
\usepackage{enumerate}
\usepackage{epsf}
\usepackage{psfrag}
\usepackage{amsmath}
\usepackage[all]{xy}
\usepackage[dvips]{graphicx}
\usepackage{epsfig}
\usepackage{epstopdf}
\usepackage{ifpdf}
\usepackage[UKenglish]{babel}
\usepackage{subfigure}
\usepackage{hyperref}
\usepackage{color}
\usepackage{marginnote}

\hypersetup{
linkbordercolor={1 0 0}, 
citebordercolor={0 1 0} 
}
\newtheorem{thm}{Theorem}[section]
\newtheorem{lemma}[thm]{Lemma}
\newtheorem{prop}[thm]{Proposition}
\newtheorem{cor}[thm]{Corollary}

\newtheorem{conj}[thm]{Conjecture}
\newtheorem{innercustomthm}{Conjecture}
\newenvironment{customthm}[1]
  {\renewcommand\theinnercustomthm{#1}\innercustomthm}
  {\endinnercustomthm}
 \theoremstyle{definition}
\newtheorem{defn}[thm]{Definition}
\theoremstyle{remark}
\newtheorem{rmk}[thm]{Remark}

\newcommand{\nb}{\textup{nb}}
\newcommand{\rank}{\textup{rk }}

\newcommand{\Z}{\mathbb{Z}}
\numberwithin{equation}{subsection}
\numberwithin{enumi}{subsection}

\makeatletter

\newcommand{\Rmnum}[1]{\expandafter\@slowromancap\romannumeral #1@}
\makeatother

\author[Jennifer Hom]{Jennifer Hom}
\address {Department of Mathematics, Columbia University, New York, NY 10027}
\email{hom@math.columbia.edu}

\author[Tye Lidman]{Tye Lidman}
\address {Department of Mathematics, The University of Texas, Austin, TX 78712}
\email{tlid@math.utexas.edu}

\author[Faramarz Vafaee]{Faramarz Vafaee}
\thanks{}
\address {Department of Mathematics, Michigan State University, East Lansing, MI 48824}
\email{vafaeefa@msu.edu}

\begin{document}
\title{Berge-Gabai knots and L-space satellite operations}
\date{}
\maketitle

\begin{abstract}
Let $P(K)$ be a satellite knot where the pattern, $P$, is a Berge-Gabai knot (i.e., a knot in the solid torus with a non-trivial solid torus Dehn surgery), and the companion, $K$, is a non-trivial knot in $S^3$. We prove that $P(K)$ is an L-space knot if and only if $K$ is an L-space knot and $P$ is sufficiently positively twisted relative to the genus of $K$. This generalizes the result for cables due to Hedden \cite{Hedden2009} and the first author \cite{Hom2011a}.
\end{abstract}

\section{Introduction} \label{sec:introduction}
In \cite{Ozsvath2013}, Oszv\'{a}th and Szab\'{o} introduced Heegaard Floer theory, which produces a set of invariants of three- and four-dimensional manifolds. One example of such invariants is $\widehat{HF}(Y)$, which associates a graded abelian group to a closed $3$-manifold $Y$. When $Y$ is a rational homology three-sphere, $\rank \widehat{HF}(Y) \ge |H_1(Y ; \mathbb{Z})|$ \cite{Ozsvath2004a}. If equality is achieved, then $Y$ is called an \emph{L-space}. Examples include lens spaces, and more generally, all connected sums of manifolds with elliptic geometry \cite{Ath}. L-spaces are of interest for various reasons. For instance, such manifolds do not admit co-orientable taut foliations \cite[Theorem 1.4]{Ozsvath2004b}.

A knot $K \subset S^3$ is called an \emph{L-space knot} if it admits a positive L-space surgery. Any knot with a positive lens space surgery is then an L-space knot. In \cite{Berge}, Berge gave a conjecturally complete list of knots that admit lens space surgeries, which includes all torus knots \cite{Moser1971}. Therefore it is natural to look beyond Berge's list for L-space knots. In \cite{Vafaee2013}, the third author classifies the twisted $(p, kp \pm 1)$-torus knots admitting L-space surgeries, some of which are known to live outside of Berge's collection.  Another related goal is to classify the satellite operations on knots that produce L-space knots. By combining work of Hedden \cite{Hedden2009} and the first author \cite{Hom2011a}, the $(m, n)$-cable of a knot $K \subset S^3$ is an L-space knot if and only if $K$ is an L-space knot and $n/m \ge 2 g(K) - 1$. (Here, $m$ denotes the longitudinal winding.) We generalize this result by introducing a new L-space satellite operation using Berge-Gabai knots \cite{Gabai1990} as the pattern. 

\begin{defn}\label{def:bg}
A knot $P \subset S^1 \times D^2$ is called a \emph{Berge-Gabai knot} if it admits a non-trivial solid torus filling.\footnote{Berge-Gabai knots, in the literature, are defined to be 1-bridge braids in solid tori with non-trivial solid tori fillings. We relax that definition to include torus knots as a proper subfamily.}
\end{defn}

To see that this satellite operation is a generalization of cabling, it should be noted that any torus knot with the obvious solid torus embedding is a Berge-Gabai knot \cite{Seifert1933}. Note also that any Berge-Gabai knot $P$ which is isotopic to a positive braid, when considered as a knot in $S^3$, admits a positive lens space surgery; for if performing appropriate surgery on $P$ in one of the solid tori in the genus one Heegaard splitting of $S^3$ returns a solid torus, then the corresponding surgery on the knot in $S^3$ will result in a lens space.  For positive braids, this surgery is positive by Lemma~\ref{lem:bgcharacterization} and \cite[Proposition 3.2]{Moser1971}.

\begin{figure}[t!]
 \begin{center}
 \subfigure[]
 {
  \includegraphics[scale=.5]{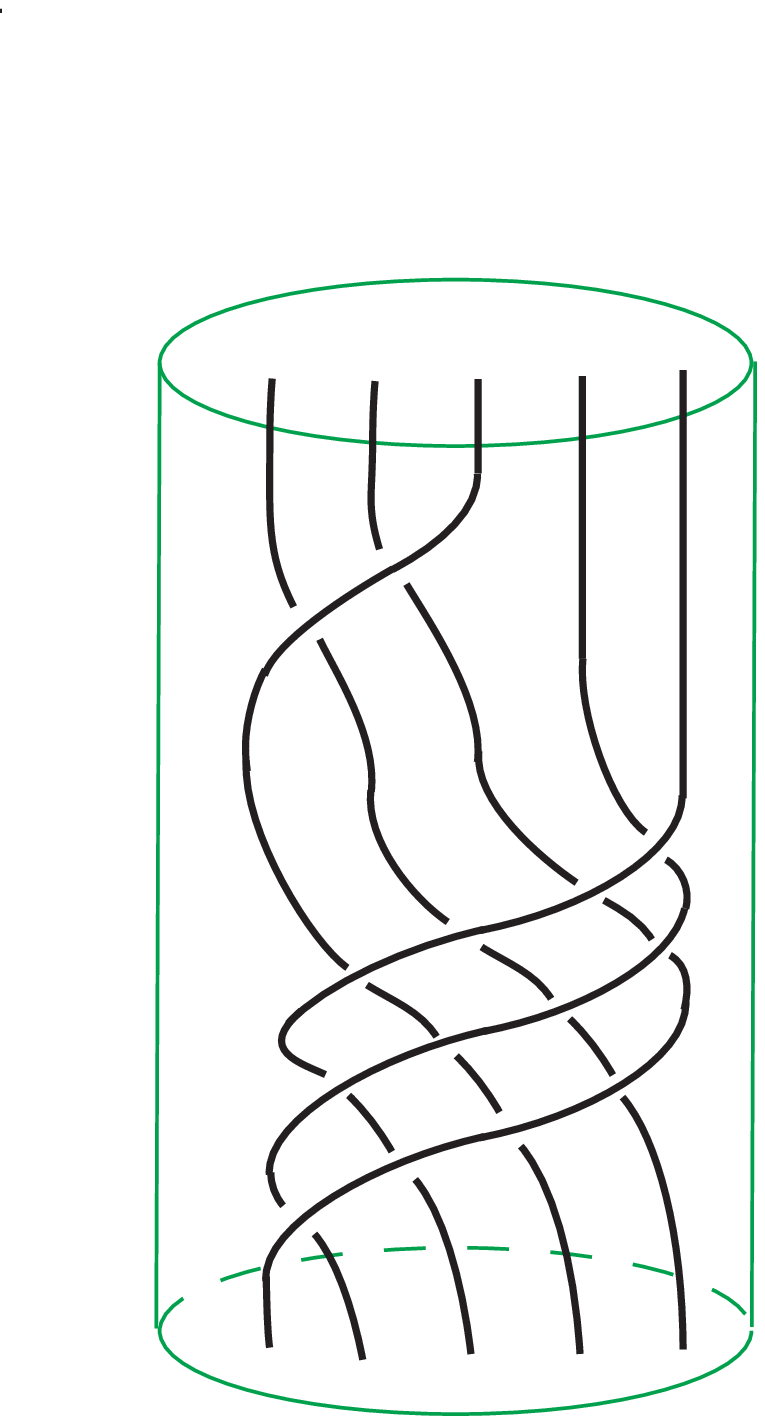}  
\label{fig1:subfig1}
   } 
\subfigure[]
{
\psfrag{l}{$\ell$}
\psfrag{m}{$m$}
\psfrag{r}{\tiny{$\Lambda$}}
\psfrag{u}{\tiny{$\mu$}}
\psfrag{A}{\small$A$}
\psfrag{D}{$D$}
  \includegraphics[scale=.5]{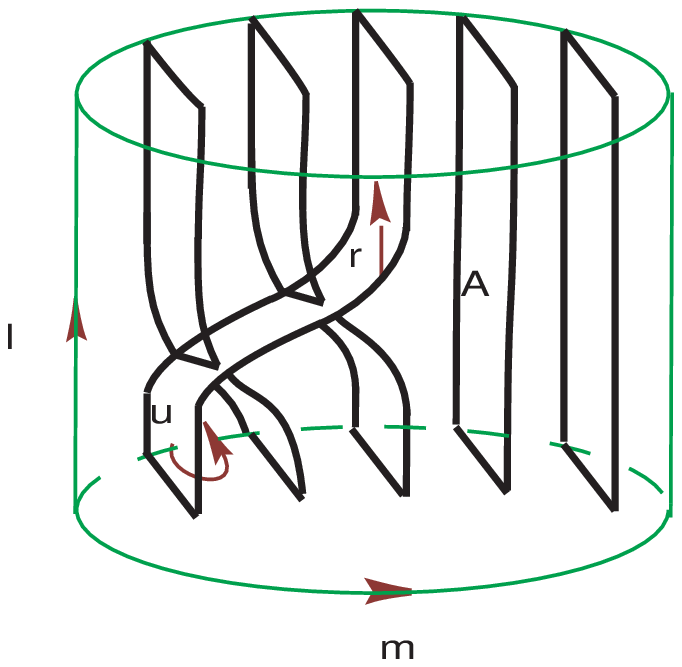}
   \label{fig1:subfig2}
   }
\caption{\small{Berge-Gabai knots are knots in $S^1 \times D^2$ with non-trivial solid tori fillings. Such knots are always the closure of the braid $(\sigma_b \sigma_{b-1} \ldots \sigma_1)(\sigma_{w-1} \sigma_{w-2} \ldots \sigma_1)^t$ where $0 \le b \le w-2$, and $|t| \ge 1$. (a) An example of a braid in a solid cylinder $I \times D^2$ that closes to form a Berge-Gabai knot with $b = 2$, $t = 3$, and $w = 5$. (The fact that the picture depicted above represents a Berge-Gabai knot is verified in \cite[Example~3.8]{Gabai1990}.) Recall that we write $t=t_0+qw$, where here $t_0=3$ and $q=0$. (b) An immersed annulus $A$ that can be arranged to be an embedded surface in $V = S^1 \times D^2$ joining $P$ to $T=\partial V$ by performing oriented cut and paste and adding a $2\pi t/w$ twist. Note that the embedded surface $A$ provides, in the exterior of $P$, a homology from $w \ell + t m$ in $T$ to $\Lambda$ in $J = \partial \text{nb}(P)$.}}
\label{fig1}
\end{center}
\end{figure}
It is shown in \cite{Gabai1989} that any Berge-Gabai knot must be either a torus knot or a 1-bridge braid in $S^1 \times D^2$.  More precisely, every Berge-Gabai knot $P \subset V = S^1 \times D^2$ is necessarily of the following form. (For a sufficient condition determining when a knot of this form is a Berge-Gabai knot, see \cite[Lemma 3.2]{Gabai1990}.) In the braid group $B_w$, where $w$ is an integer with $w \ge 2$, let $\sigma_i$ denote the generator of $B_w$ that performs a positive half twist on strands $i$ and $i+1$. Let $\sigma = \sigma_b \sigma_{b-1} \ldots \sigma_1$ be a braid in $B_w$ with $0 \le b \le w-2$  and let $t$ be a nonzero integer. Place $\sigma$ in a solid cylinder and glue the ends by a $2 \pi t/w$ twist, i.e., form the closure of the braid word $(\sigma_b \sigma_{b-1} \ldots \sigma_1)(\sigma_{w-1} \sigma_{w-2} \ldots \sigma_1)^t$. We only consider the case where this construction produces a knot, rather than a link.  This construction forms a torus knot if $b = 0$ and a 1-bridge representation of $P$ in $V$ if $1 \le b \le w-2$. We call $w$ the \emph{winding number}, $b$ the \emph{bridge width}, and $t$ the \emph{twist number} of $P$. Note that the twist number can be written as $t = t_0 + q w$ for some integers $t_0$ and $q$ where $t_0$ can be chosen so that $1 \le t_0 \le w-1$.\footnote{Our construction of Berge-Gabai knots, which enables us to define them up to isotopy of the knot in $S^1 \times D^2$, is slightly different than that of Gabai~\cite{Gabai1990}. In Gabai's original construction, he always took $q = 0$ and considered knots in the solid torus up to homeomorphism of $S^1 \times D^2$  taking one knot to the other.}  See Figure \ref{fig1:subfig1}.  Also, note that if $b \ne 0$ then the possibility of $t_0 = w-1$ is disallowed as otherwise we would obtain a link with at least two components \cite{Gabai1990}. 

\begin{rmk}\label{rmk}
Note that if $t<0$, then the braid $\sigma = (\sigma_b \sigma_{b-1} \ldots \sigma_1)(\sigma_{w-1} \sigma_{w-2} \ldots \sigma_1)^t$ is isotopic to a negative braid:
\begin{align*}
	\sigma &\sim (\sigma_b \sigma_{b-1} \ldots \sigma_1)(\sigma_{w-1} \sigma_{w-2} \ldots \sigma_1)^t \\
		&\sim (\sigma_{w-1} \sigma_{w-2} \ldots \sigma_{b+1})^{-1}(\sigma_{w-1} \sigma_{w-2} \ldots \sigma_1)^{t+1}.
\end{align*}
\end{rmk}
 
We are now ready to state the main result. Let $P(K)$ denote a satellite knot with pattern $P$ and companion $K$.  
{\thm \label{thm}Let $P$ be a Berge-Gabai knot with bridge width $b$, twist number $t$, and winding number $w$, and let $K$ be a non-trivial knot in $S^3$. Then the satellite $P(K)$ is an L-space knot if and only if $K$ is an L-space knot and  $\frac{b+ t w}{w^2} \ge 2 g(K) - 1.$} 
\\

Note that when $b = 0$, we can take $w = m$ and $t = n$, and Theorem \ref{thm} reduces to the cabling result of \cite{Hedden2009, Hom2011a}. A version of the ``if" direction of Theorem~\ref{thm} appears in \cite[Proposition~7.2]{motegi2014}.

The outline of the proof of Theorem \ref{thm} is as follows. By applying techniques developed in \cite{Gabai1990, Gordon1983} to carefully explore the framing change of the solid torus surgered along $P$, we prove the ``if'' direction of the theorem. More precisely, surgery on $P(K)$ corresponds to first doing surgery on $P$ (namely removing a neighborhood of $P$ from $S^1 \times D^2$ and Dehn filling along the new toroidal boundary component) and, second, attaching this to the exterior of $K$.  Therefore, if one chooses the filling on $P$ such that the result is a solid torus (using that $P$ is a Berge-Gabai knot), then the overarching surgery on $P(K)$ corresponds to attaching a solid torus to the exterior of $K$ (performing surgery on $K$). Moreover, note that by positively twisting $P$ by performing a positive Dehn twist on $S^1 \times D^2$ (i.e., increasing $q$), we can obtain an infinite family of Berge-Gabai knots. Fixing an L-space knot $K$, for sufficiently large $q$, the satellite $P(K)$ will admit a positive L-space surgery. Finally, the ``only if'' direction is proved by methods similar to those used in \cite{Hom2011a}.

In order to prove Theorem \ref{thm}, we establish the following lemma, which may be of independent interest.
\begin{lemma}
\label{lem:negbraid}
Let $P \subset S^1 \times D^2$ be a negative braid and $K \subset S^3$ be an arbitrary knot. Then the satellite knot $P(K)$ is never an L-space knot.
\end{lemma}

\noindent We point out that Lemma~\ref{lem:negbraid} can be extended more generally to the case that $P$ is a homogeneous braid which is not isotopic to a positive braid \cite[Theorem 2]{Stallings1978}.  The proof of Lemma~\ref{lem:negbraid} was inspired by the arguments in \cite{BM}.  

We have the following corollary concerning the Ozsv\'ath-Szab\'o concordance invariant $\tau$ and the smooth $4$-ball genus.

\begin{cor}
Let $P \subset S^1 \times D^2$ be a Berge-Gabai knot and $K \subset S^3$ be an L-space knot. If $\frac{b+ t w }{w^2} \geq 2 g(K) - 1$, then
\[ \tau(P(K)) = \tau(P) + w \tau(K), \]
and
\[ g_4(P(K)) = g_4(P) + w g_4(K), \]
where $\tau(P)$, respectively $g_4(P)$, denotes $\tau$, respectively the $4$-ball genus, of the knot obtained from the standard embedding of $S^1 \times D^2$ into $S^3$.
\end{cor}

\begin{proof}
If $J$ is an L-space knot, then $\tau(J)=g_4(J)=g(J)$ by \cite[Corollary 1.3]{Ni2009} and \cite[Corollary 1.6]{Ath}. Furthermore, by Lemma \ref{lem3},
\[ g(P(K)) = g(P) + wg(K).\]
By assumption, $K$ is an L-space knot.  The result is clear if $K$ is trivial, so assume that $K$ is non-trivial.  Since $P$ is a Berge-Gabai knot with a necessarily positive twist number, it follows that $P$ is isotopic to a positive braid. Therefore, by the discussion following Definition~\ref{def:bg}, $P$ has a positive lens space surgery, and thus is an L-space knot.  Furthermore, by Theorem~\ref{thm}, we also have that $P(K)$ is an L-space knot, and the result follows.
\end{proof}

Theorem~\ref{thm} allows one to construct new examples of L-spaces as follows.  First, begin with any L-space knot and then satellite with a Berge-Gabai knot satisfying the conditions in Theorem~\ref{thm}.  Sufficiently large positive surgery will then result in an L-space.  Using this technique, we will construct L-spaces with any numbers of hyperbolic and Seifert fibered pieces in the JSJ decomposition. 

\begin{thm}\label{thm:jsj}
Let $r$ and $s$ be non-negative integers such that at least one is non-zero.  Then there exist infinitely many irreducible L-spaces whose JSJ decompositions consist of exactly $r$ hyperbolic pieces and $s$ Seifert fibered pieces.
\end{thm}

As discussed, an L-space cannot admit a co-orientable taut foliation. Therefore, Theorem~\ref{thm:jsj} will yield irreducible rational homology spheres without co-orientable taut foliations whose JSJ decompositions consist of any numbers of hyperbolic and Seifert fibered pieces.  We remark that all rational homology spheres with Sol geometry are L-spaces \cite{BGW2012}.

It is also natural to ask in what sense Theorem~\ref{thm} generalizes; in particular, given a satellite knot which is an L-space knot, what must hold for the pattern or the companion?  We propose the following conjecture (see also \cite[Question 22]{BM}).

\begin{conj}\label{conj:satlspace}
If $P(K)$ is an L-space knot, then so are $K$ and $P$.  
\end{conj}

Similarly, we conjecture that the converse holds as well, contingent on the pattern being embedded ``nicely'' in the solid torus (e.g., as a strongly quasipositive braid closure) and sufficiently ``positively twisted'' (akin to the condition in Theorem~\ref{thm}).  We will not attempt to make these notions precise in this paper.  

As supporting evidence for Conjecture~\ref{conj:satlspace}, we will study it from the viewpoint of left-orderability.  Recall that a non-trivial group $G$ is {\em left-orderable} if there exists a left-invariant total order on $G$ (see Section~\ref{sec:furtherresults} for a more detailed discussion).  We recall the conjecture of Boyer, Gordon, and Watson relating Heegaard Floer homology to the left-orderability of three-manifold groups.

\begin{conj}[Boyer-Gordon-Watson \cite{BGW2012}]\label{conj:lolspace}
Let $Y$ be an irreducible rational homology sphere.  Then $Y$ is an $L$-space if and only if $\pi_1(Y)$ is not left-orderable.  
\end{conj}   

We point out that the computational strengths of Heegaard Floer homology and left-orderability tend to be fairly different.  It is hopeful that if Conjecture~\ref{conj:lolspace} is true then the strengths of each theory could be combined to derive new topological consequences.  We utilize this philosophy to establish Conjecture~\ref{conj:satlspace} under the assumption of Conjecture~\ref{conj:lolspace}.  

\begin{prop}\label{prop:losatellite}
Assuming Conjecture~\ref{conj:lolspace}, if $P(K)$ is an L-space knot, then so are $P$ and $K$.  
\end{prop}

\subsection*{Acknowledgements}

We would like to thank Matthew Hedden for helpful discussions and his interest in our work. We are also grateful to Josh Greene for pointing out Remark~\ref{rmk}, to Allison Moore, David Shea Vela-Vick, and Rachel Roberts for help with the proof of Lemma~\ref{lem:negbraid}, and to Ko Honda for a helpful discussion.  The first author was partially supported by NSF grant DMS-1307879. The second author was partially supported by NSF grant DMS-0636643.

\section{The main result}\label{sec:mainsurgery}
In this section, we provide background on 1-bridge braids in solid tori and Dehn surgery on satellite knots. See \cite{Berge1991, Gabai1990, Gordon1983} for further details. Throughout the rest of the paper, we assume that $P$ is a Berge-Gabai knot in $V = S^1 \times D^2$ (i.e., $P$ admits a non-trivial solid torus surgery) unless otherwise stated. We also consider the standard embedding of $S^1 \times D^2$ into $S^3$ such that $S^1 \times \{*\}$ bounds an embedded disk in $S^3$. When it is clear from context, we will not distinguish between the Berge-Gabai knot $P \subset V$ and $P \subset S^3$.

\subsection{Berge-Gabai knots}\label{sec2:subsec1} The primary goal of this subsection is to highlight the Dehn surgeries on $P \subset V$ that will return a solid torus. In what follows, we provide a setup similar to that of \cite{Gabai1990}.  

An arbitrary knot $P$ in $V$ is called a \emph{1-bridge braid} if $P$ can be isotoped to be a braid in $V$ that lies in $S^1 \times \partial D^2$ except for one arc that is properly embedded in $V$, and $P$ is not a torus knot. Gabai \cite{Gabai1989} showed that any knot in a solid torus with a non-trivial solid torus surgery must be either a torus knot or a 1-bridge braid in $S^1 \times D^2$, and Berge \cite{Berge1991} classified all 1-bridge braids in $S^1 \times D^2$ with non-trivial solid tori fillings. We denote the braid index of $P$ by $w$.

We will consider $\widehat{V}$, the exterior of $P \subset V$. Let $T = \partial V$ and $J = \partial \text{nb}(P)$. We equip $T$  with the homological generators $(m, \ell)$ where $\ell$ is the longitude $S^1 \times \left\{*\right\}$ of $T$ and $m$ is $\left\{*\right\} \times \partial D^2$; therefore, $\ell$ becomes null-homologous after standardly embedding $V$ in $S^3$ and removing $\text{nb}(P)$. 
We equip $J$ with homological generators $(\mu, \Lambda)$ as follows. The generator $\mu$ is the meridian of $P$. Note that $m$ is homologous to $w \mu$ in $\widehat V$. To define $\Lambda$, consider the immersed annulus $A$ connecting $J$ to $T$ with $b$ arcs of self-intersection in Figure \ref{fig1:subfig2}. By doing oriented cut and paste to the arcs of self-intersection we can arrange $A$ to be an embedded surface in $\widehat{V}$ joining $J$ to $T$. Define $\Lambda$ to be $A  \cap J$.
Orient $m$, $\ell$, $\mu$, and $\Lambda$ as in Figure \ref{fig1:subfig2}. Note that $A \cap T= w \ell + t m$, and so $w \ell + t m$ is homologous to $\Lambda$ in $\widehat{V}$.

Let $\lambda$ be the simple closed curve on $J$ that is homologous to $\Lambda - w t \mu \in H_1(J; \mathbb{Z})$. Thus, we have the following equalities in $H_1(\widehat{V}; \Z)$:
\begin{align*}
	[\lambda] &= [\Lambda - wt\mu] \\
		&= [w\ell + tm - wt\mu] \\
		&= [w\ell],
\end{align*}
where the last equality follows from the fact that $m$ is homologous to $w\mu$. In particular, $\lambda$ becomes null-homologous after standardly embedding $V$ in $S^3$ and removing $\text{nb}(P)$.  Now the equation $[\lambda] = [\Lambda - w t \mu]$ can be used to switch from $(\mu, \Lambda)$- to $(\mu, \lambda)$-coordinates, where $(\mu, \lambda)$ are the usual meridian-longitude coordinates on $P$ when $V$ is standardly embedded in $S^3$. 

We recall that a 1-bridge braid in $S^1 \times D^2$ with winding number $w$, bridge width $b$, and twist number $t$ can be represented via the braid word $\sigma = (\sigma_b \sigma_{b-1} ... \sigma_1)(\sigma_{w-1} \sigma_{w-2} ... \sigma_1)^t$ where $|t| \ge 1$, and $1 \le b \le w-2$. The following lemma is a consequence of \cite[Lemma 3.2]{Gabai1990}:

{\lemma\label{lem:bgcharacterization} Let $P$ be a 1-bridge braid in $V$ and $s$ a positive integer. If filling $\widehat{V}$ along a curve $\alpha = d \mu + s \Lambda$ in $J$ yields $S^1 \times D^2$, then $s = 1, d \in \{ b, b + 1 \}$, and $\gcd(w, d) = 1$.

In $(\mu, \lambda)$-coordinates these possible exceptional surgeries are $\alpha =  (t w + d)\mu + \lambda$ where $d \in  \{ b, b + 1 \}$.}\footnote{We have stated Lemma~\ref{lem:bgcharacterization} so that the orientation of $(\mu, \lambda)$ agrees with the standard convention that $\mu \cdot \lambda = 1$. In Gabai's paper \cite{Gabai1990}, $\mu$ is oriented opposite to that of Figure \ref{fig1:subfig2}.}
\\

Note that when $P$ is an $(m,n)$-torus knot in $V$, there are infinitely many surgeries on $P$ that will return a solid torus, including $mn+1 = tw + b + 1$; this follows, for instance, from the proof of \cite[Proposition 3.2]{Moser1971}.

Let $(P; n_1/n_2)$ denote the result of filling $\widehat{V}$ along the curve $n_1 \mu + n_2 \lambda$. Lemma \ref{lem:bgcharacterization} shows that if $P$ is a Berge-Gabai knot, then $(P; p_d)$ will be homeomorphic to $S^1 \times D^2$ for at least one of the coefficients $\displaystyle p_d = t w + d, d \in \{ b, b + 1 \}$.

Note that adding a positive full-twist to all of the $w$ strands of $P$ results in a new knot $P^{'}$ where $t$ changes into $t+w$. Correspondingly, there exists a homeomorphism of the solid torus (doing a positive meridional twist), which takes $P$ to $P^{'}$. Iterating this process $q$ times, we get the following:

{\prop \label{prop1}Let $P$ be a Berge-Gabai knot in $S^1 \times D^2$, standardly embedded in $S^3$, so that $(P; p)$ is homeomorphic to a solid torus. Let $P^{'}$ be the knot obtained from $P$ by adding $q$ positive Dehn twists. Then
\[
 (P^{'}; p+ q w^2) \cong S^1 \times D^2.
\]}

Hence if we have a Berge-Gabai knot $P$ with twist number $t$, adding $q$ full twists to all $w$ strands of $P$ will produce a Berge-Gabai knot with twist number $t + q w$.

\subsection{Surgery on $P(K)$}\label{subsec:surgery}
Let $P(K)$ be a satellite knot with pattern $P \subset V$ and companion $K$. Let $f : V  \rightarrow \nb(K)$ be a homeomorphism that determines the zero framing of $K$, i.e., $[f(S^1 \times \{*\} )] = 0 \in H_1(X; \mathbb{Z})$ where $X = S^3 - \text{nb}(K)$. Thus $P(K) = f(P).$ 

Recall that $m, \ell \in H_1(T; \mathbb{Z})$ are the natural meridian and longitude coordinates of $T= \partial V$, oriented such that $m \cdot \ell = 1$. Recall also that $\widehat{V} = V -\text{nb}(P)$. Note that $H_1(\widehat{V}) = \mathbb{Z}\langle \ell \rangle \oplus \mathbb{Z}\langle \mu\rangle$ where $\mu$ is the class of the meridian of $\text{nb}(P)$. When $P$ is viewed as a knot in $S^3$, let $\lambda \subset \partial \text{nb}(P)$ be the unique curve on $\partial \text{nb}(P)$ which is null-homologous in $S^3 - \text{nb}(P)$ (i.e., the zero framing of $P$).  That is, if $f$ is as above, then $f(\lambda)$ is the zero framing of $P(K)$. Thus, $S^{3}_{p_1/p_2}(P(K)) \cong X \cup_f (P ; p_1/p_2)$, where the notation means $\partial X$ and $\partial (P ; p_1/p_2)$ are identified via the restriction of $f$ to $\partial (P ; p_1/p_2) = \partial V$. With the above notation:

\begin{lemma}[{\cite[Lemma 3.3]{Gordon1983}}]\label{lem2} For relatively prime integers $p_1, p_2$, and $P \subset V$ with winding number $w$:
\begin{itemize}
 \item[(a)]\label{lem2:H1}  $H_1((P ; p_1/p_2); \mathbb{Z}) \cong \mathbb{Z} \oplus \mathbb{Z}_{\gcd(w, p_1)}$.

\item[(b)]\label{lem2:kernel} If $w \neq 0$, the kernel of $H_1(\partial (P; p_1/p_2); \mathbb{Z}) \rightarrow H_1((P; p_1/p_2); \mathbb{Z})$ is the cyclic group generated by $$\displaystyle \frac{p_1}{\gcd(w, p_1)} m + \frac{p_2 w^2}{\gcd(w, p_1)} \ell .$$
\end{itemize}
\end{lemma}

Note that Lemma \ref{lem2} is valid regardless of whether or not $P$ is a Berge-Gabai knot. However, when $P$ is a Berge-Gabai knot, we can use Lemma \ref{lem2} to relate surgeries on $K$ and $P(K)$ in the following sense. 

{\cor \label{cor1}Let $P$ be a Berge-Gabai knot in $V$ with winding number $w$ so that $(P; p) \cong S^1 \times D^2$. Then
\[ \displaystyle S^3_p(P(K)) \cong S^3_{p/{w^2}}(K).\]}

\begin{proof}
The result essentially follows from the fact that $$S^{3}_{p}(P(K)) \cong X \cup_f (P ; p).$$ \noindent By assumption, $(P;p)$ is homeomorphic to a solid torus. Therefore, in order to find the corresponding surgery coefficient on $K$, one needs to determine the slope of the meridian of $\partial (P; p)$ under the canonical identification with $\partial V$, and where it is sent under $f$.  

Note that the slope of the meridian of $(P;p)$ is precisely the generator of $$\ker\Big( H_1(\partial (P; p); \mathbb{Z}) \rightarrow H_1((P; p); \mathbb{Z})\Big).$$  
\noindent Using the identification of $\partial V$ and $\partial (P;p)$, we have that the slope of the meridian, in $(m,\ell)$-coordinates, is given by $(p,w^2)$ by Lemma~\ref{lem2}.  Since $f$ sends $m$ (respectively $\ell$) to the meridian (respectively longitude) of $K$, the result follows.
\end{proof}

\noindent Combining Lemma~\ref{lem:bgcharacterization} with Corollary~\ref{cor1}, we deduce the following:

{\prop \label{prop2}Let $P$ be a Berge-Gabai knot with bridge width $b\ne 0$, winding number $w$, and twist number $t$, and let $K$ be an arbitrary knot in $S^3$. Then for at least one $d \in \{ b, b + 1 \}$, 

\[ S^{3}_{d + tw}(P(K)) \cong  S^{3}_{\frac {d + tw}{w^2}}(K). \]
}

Note that $\gcd(d+tw, w^2)=1$ (see Lemma~\ref{lem:bgcharacterization}). We end this subsection by stating the following lemma, which turns out to be useful during the course of proving Theorem \ref{thm}. Let $\Delta_K(T)$ denote the symmetrized Alexander polynomial of $K$. Recall the behavior of the Alexander polynomial for satellites (see for instance \cite{Lickorish1997}): 
\begin{equation}\label{eqn:alexandersatellite}
\Delta_{P(K)}(T) = \Delta_{P}(T) \Delta_{K}(T^w).
\end{equation}
{\lemma \label{lem3} Let $P(K)$ be a fibered satellite knot where $P$ has winding number $w$. Then $$g(P(K)) = g(P) + w g(K).$$
\noindent Furthermore, if $P$ is a Berge-Gabai knot as above with $t>0$, then 
\begin{equation} \label{eqn:gP}
	g(P) = \frac {(t - 1)(w - 1) + b}{2}.
\end{equation}
}
\begin{proof}

Since $P(K)$ is a fibered knot, we deduce that $\deg \Delta_{P(K)}(T) = g(P(K))$. It also follows that $K$ and $P$ are both fibered \cite{Hirasawa2008}. Combining these two facts with \eqref{eqn:alexandersatellite}, we see that $g(P(K)) = g(P) + w g(K)$. 

In order to calculate $g(P)$, notice that $P$ is a positive braid if $t>0$. Hence, the Seifert surface $R$ obtained from Seifert's algorithm is a minimal genus Seifert surface for $P$ \cite{Stallings1978}. Then
\[ \chi (R) = 1 - 2 g(P) \Rightarrow w - b - t(w - 1) = 1 - 2 g(P). \]
\end{proof}

\subsection{Input from Heegaard Floer theory} \label{sec:HFinput}
In this subsection we mainly use the notation of \cite{Hom2011a}. Recall that an \emph{L-space} $Y$ is a rational homology sphere with the simplest possible Heegaard Floer homology, i.e., $\rank \widehat{HF}(Y) = |H_1(Y ; \mathbb{Z})|$. We say that a knot $K$ in $S^3$ is an \emph{L-space knot} if it admits a positive L-space surgery.  

We let $\tau(K)$ denote the integer-valued concordance invariant from \cite{Ozsvath2003}. Let $\mathcal{P}$ denote the set of all knots $K$ for which $g(K) = \tau(K)$.  (Recall from \cite{Hedden2010} that for fibered knots, $g(K) = \tau(K)$ is equivalent to being strongly quasipositive.) If $K$ is an L-space knot, then $K \in \mathcal{P}$. This follows from  \cite[Corollary 1.6]{Ath} and the fact that L-space knots are fibered \cite[Corollary 1.3]{Ni2009}.

Let
\[
s_{K} = \sum_{i \in \mathbb{Z}} \left( \rank  H_*(\widehat{A}^K_{i}) - 1\right),
\]
where $\widehat{A}^K_{i}$ is the subquotient complex of $CFK^{\infty}(K)$ defined in \cite{Ozsvath2008}. It is proved in \cite{Hom2011a} that $\rank  H_*(\widehat{A}^K_{i})$ is always odd, and so $s_{K}$ is always a non-negative even integer.  For a pair of relatively prime non-zero integers $m$ and $n$, $n>0$, let
\begin{equation}\label{nu-equation}
t^{m/n}_{K} = 2 \max(0, n(2 \nu(K) - 1) - m).
\end{equation}
Observe that
\begin{equation}\label{eqn:tmn}
t^{m/n}_{K} = 0 \quad  \text{ if and only if } \quad m/n \ge 2 \nu(K) - 1.
\end{equation}
The term $\nu(K)$ is another integer-valued invariant of $K$, defined in \cite[Definition 9.1]{Ozsvath2010}, which is bounded below by $\tau(K)$ and above by $g(K)$. In particular, if $K \in \mathcal{P}$, then $\nu(K) = g(K)$.  

Let $m$ and $n$ be as above, and suppose that $\nu(K) \ge \nu(\overline{K})$ where $\overline{K}$ denotes the mirror of $K$.  (This condition is automatically satisfied for $K \in \mathcal{P}$.) If $\nu(K)>0$ or $m>0$, then
\begin{equation}\label{eqn:rankformula}
\rank \widehat{HF}(S^3_{m/n}(K)) = m + n s_{K} + t^{m/n}_K
\end{equation}
by \cite[Proposition 9.6]{Ozsvath2010}.

By \eqref{eqn:rankformula}, when $m>0$ we have that
\begin{equation} \label{eqn:ts}
	S^3_{m/n}(K) \textup{ is an L-space if and only if } t^{m/n}_K=0 \textup{ and } s_{K} = 0.
\end{equation}
By \cite[Theorem 4.4]{Ozsvath2004}, the group $H_*(\widehat{A}^K_{i})$ is isomorphic to $\displaystyle \widehat{HF}(S^3_N(K), [i])$ for $N \gg 0$ and $|i| \leq N/2$. Thus, we have that
\begin{equation} \label{eqn:s}
	K \textup{ is an L-space knot if and only if } s_K=0.
\end{equation}
In fact, if $K$ is a non-trivial L-space knot, $S^3_{m/n}(K)$ is an L-space if and only if $m/n \geq 2g(K) - 1$.  This follows from \eqref{eqn:tmn}, \eqref{eqn:rankformula}, and the fact that for $K$ a non-trivial L-space knot, $\nu(K) = g(K) > 0$. (The original argument for the forward direction is given in \cite{Kronheimer2007}.)

\subsection{Proof of Theorem \ref{thm}}
This subsection is devoted to the proof of Theorem \ref{thm}.  We begin with the proof of Lemma~\ref{lem:negbraid}. We do not review the concept of a quasipositive Seifert surface but instead refer the reader to \cite{Hedden2010, Rudolph}.

\begin{proof}[Proof of Lemma~\ref{lem:negbraid}]
Suppose for contradiction that $P(K)$ is an L-space knot.  Recall that L-space knots are fibered \cite{Ni2009, Ath}.    
It is also a well-known fact that a minimal genus Seifert surface for a negative braid can be expressed as a plumbing of negative Hopf bands \cite[Theorem~2]{Stallings1978}. (See also \cite[Theorem~1]{akbulut2001lefschetz} for an explicit construction in the case of torus knots.) Since $P(K)$ is fibered, this implies that $K$ is fibered and $P$ is fibered in the solid torus \cite{Hirasawa2008}, so the fiber for $P(K)$ is constructed by patching the fiber for $P$ in the solid torus to $w$ copies of the fiber for $K$. As a result, when $P$ is a negative braid, the fiber surface for $P(K)$ contains (at least) as many negative Hopf bands as the one for $P$. 

By the above description of the fiber surface, we can deplumb a negative Hopf band.  This means we can decompose the fiber surface for $P(K)$ as a Murasugi sum, where one of the summands is not a quasipositive surface.  By \cite{Rudolph}, if a Seifert surface is a Murasugi sum, it is quasipositive if and only if all of the summands are quasipositive.  Thus, the fiber surface for $P(K)$ is not a quasipositive surface.  However, since $P(K)$ is an L-space knot, it is strongly quasipositive \cite{Hedden2010}, which gives a contradiction.     
\end{proof}  

We prove Theorem~\ref{thm} only for the cases where $b \ne 0$ (consequently $1 \le t_0 \le w-2$) and refer the reader to \cite{Hedden2009, Hom2011a} for the case $b=0$. 
\begin{proof}[Proof of Theorem \ref{thm}]
$(\Leftarrow)$ The proof of this direction follows from Proposition \ref{prop2}, which tells us that
\[ S^3_{d+tw}(P(K)) \cong S^3_{\frac{d+tw}{w^2}}(K). \]
Since $K$ is a non-trivial L-space knot and $\frac{b+tw}{w^2} \geq 2g(K)-1 > 0$, it follows that $S^3_{\frac{d+tw}{w^2}}(K)$ is an L-space.  Here we are using that $d \geq b$.  Therefore, $P(K)$ is an L-space knot.

$(\Rightarrow)$ For the case that $t<0$ (see Remark~\ref{rmk}), we apply Lemma~\ref{lem:negbraid} to see that $P(K)$ cannot be an L-space knot.  Therefore, we can assume that $t>0$ and $P(K)$ is an L-space knot. For simplicity of notation, we set $m = d + t_0 w + q w^2$ where $d \in \{ b, b + 1 \}$ is such that $(P;m) \cong S^1 \times D^2$.  Again from Proposition \ref{prop2} we have 
\begin{equation}\label{eqn:rankequality}
\rank \widehat{HF}(S^{3}_{m}(P(K))) = \rank\widehat{HF}(S^{3}_{m/w^2}(K)).
\end{equation}
Since $P(K)$ is an L-space knot, it follows that $g(P(K)) = \tau(P(K))$, and we see that
\begin{equation}\label{eqn:tP} 
t^{m}_{P(K)} = 2 \max(0, 2 g(P(K)) - 1 - m).
\end{equation} 

We first suppose that $\nu(K) \ge \nu(\overline{K})$.  Since $m>0$, we may combine \eqref{eqn:rankformula}, \eqref{eqn:s}, and \eqref{eqn:rankequality} to obtain 
\[ 
m + t^{m}_{P(K)} = m + w^2 s_{K} + t^{m/w^2}_K ,
\]
or equivalently
\begin{equation} \label{eqn:t}
t^{m}_{P(K)} = w^2 s_{K} + t^{m/w^2}_K.
\end{equation}
Note that by Lemma~\ref{lem3}, \eqref{eqn:gP}, and \eqref{eqn:tP}, we have that 
\begin{equation}\label{eqn:tg}
t^m_{P(K)} = \max (0,4wg(K) - 2w - 2t_0 - 2qw + 2b - 2d).  
\end{equation}
\emph{Claim}. The equality in \eqref{eqn:t} does not hold unless  both sides are identically zero.
\begin{proof}[Proof of the Claim]
If $t^{m}_{P(K)} \ne 0$ then we have two cases:

\begin{itemize}
	\item[Case 1.] Suppose $t^{m/w^2}_K = 0$. Using \eqref{eqn:tg}, we see \eqref{eqn:t} is equivalent to
 \[4 w g(K) - 2 w - 2 t_0 - 2 q w +2b - 2d= w^2 s_K. \]
It follows that $w$ divides $2t_0+2d-2b$. Since $d-b \in \left \{ 0, 1 \right \}$ and $1 \leq t_0 \leq w-2$, we conclude that $w=2t_0+2d-2b$. Since $$4 w g(K) - 2 w - w - 2 q w = w^2 s_K,$$ then
\[ 4 g(K) - 3 - 2 q  = w s_K. \]
The right side is an even number and the left side is odd which is a contradiction.
	\item[Case 2.] Suppose $t^{m/w^2}_K \ne 0$. By expanding both sides of \eqref{eqn:t} and again using \eqref{eqn:tg}, we see that
\[
       4 w g(K) - 2 w - 2 t_0  - 2 q w + 2b - 2d= w^2 s_K + 4 w^2 \nu(K) - 2 w^2 - 2 d - 2 t_0 w - 2 q w^2.
\]
By rearranging terms, we get
\[
4 w g(K) - 2 w + 2 (b - t_0)  - 2 q w + 2 t_0 w =  w^2 (4 \nu(K) - 2 - 2 q + s_K). 
\]
Therefore $w$ divides $2(b-t_0)$.  Since $b$ and $t_0$ are both bounded above by $w-2$, we have either $2(b - t_0) = \pm w$ or $b = t_0$.  

Recall that we described $P$ as a braid closure in Section~\ref{sec:introduction}.  Viewing this braid as a mapping class of the disk with $w$ punctures, it is straightforward to verify that if $b = t_0$, the $(t_0+1)^\textup{th}$ puncture is fixed. Therefore, in this case $P$ has at least two components, which contradicts $P$ being a knot.  Thus, we must have $2(b - t_0) = \pm w$.  

Substituting and dividing by $w$ gives:
\[
4 g(K) - 2 \pm 1  - 2 q  + 2 t_0 = w (4 \nu(K) - 2 - 2 q + s_K).
\]
As in Case 1, comparing the parities of each side gives a contradiction.
\end{itemize}
\end{proof}
\noindent Having proved the claim, all the terms in \eqref{eqn:t} are identically zero. Since $ s_{K} =0$, \eqref{eqn:s}  gives that $K$ is an L-space knot. Also, $ t^{m}_{P(K)} = 0$ together with \eqref{eqn:tg} implies  
\begin{equation}\label{eqn:th}
 \displaystyle \frac{t_0 + q w + d - b}{w} \ge 2 g(K) - 1.
\end{equation}

Since $1 \le t_0 \le w-2$ and $(d-b) \in \left \{ 0,  1 \right \}$, we have that $0 \leq t_0 + d - b < w$. Note that $2g(K) - 1$ is an integer, so we deduce that \eqref{eqn:th} holds if and only if
\[ q \geq 2g(K)-1, \]
which implies that
\[ \displaystyle \frac{b + t_0 w + q w^2}{w^2} \ge 2 g(K) - 1, \]
as desired. 

Now suppose that $\nu(K) < \nu(\overline{K})$. We claim that in this case, $P(K)$ is not an L-space knot, which is a contradiction.  Recall from \cite[Equation (34)]{Ozsvath2010} that $\nu(K)$ is equal to either $\tau(K)$ or $\tau(K)+1$, and from \cite[Lemma 3.3]{Ozsvath2003} that $\tau(\overline{K}) = -\tau(K)$. Thus, when $\nu(K) < \nu(\overline{K})$, it follows that $\nu(\overline{K})>0$. By \cite[Proposition 2.5]{Ozsvath2004a}, the total rank of $\widehat{HF}(Y)$, for a closed three-manifold $Y$, is independent of the orientation of $Y$, i.e., 
\begin{equation}\label{eqn:orientation}
\rank \widehat{HF}(Y) = \rank \widehat{HF}(-Y).  
\end{equation}
By combining \eqref{eqn:orientation}, Proposition~\ref{prop2}, and the fact that 
\begin{equation} \label{eqn:orientationreverse}
S^3_{m/n}(K) \cong -S^3_{-m/n}(\overline{K}),
\end{equation}
we deduce that 
\begin{equation}\label{eqn:rankequality2}
\rank \widehat{HF}(S^{3}_{m}(P(K))) = \rank \widehat{HF}(S^{3}_{-m/w^2}(\overline{K})).
\end{equation}
By combining  \eqref{eqn:rankformula}, \eqref{eqn:s}, and \eqref{eqn:rankequality2}, since $P(K)$ is an L-space knot, we have  
\begin{equation} \label{eqn:t2}
m + t^{m}_{P(K)} = -m + w^2 s_{\overline{K}} + t^{-m/w^2}_{\overline{K}}.
\end{equation}
Using \eqref{nu-equation} and the fact that $\nu(\overline{K})>0$, we observe that $t^{-m/w^2}_{\overline{K}} \ne 0$. \\

\noindent \emph{Claim}. The equality in \eqref{eqn:t2} never holds.
\begin{proof}[Proof of the Claim]
We prove the claim by considering the following two cases:
\begin{itemize}
              \item[Case 1.] Suppose $t^m_{P(K)} \ne 0$. Using \eqref{eqn:tg}, by expanding both sides of \eqref{eqn:t2} we get that
\begin{align*}
&d + t_0 w +q w^2 +  4 w g(K) - 2 w - 2 t_0 - 2 q w +2b - 2d  \\
&= -d - t_0 w - qw^2 + w^2 s_{\overline{K}} +4 w^2 \nu(\overline{K}) - 2 w^2 + 2 d + 2 t_0 w + 2 q w^2.
\end{align*}
A similar reasoning as in Case 1 of the previous part of the proof shows that this equality gives a contradiction.
            \item[Case 2.] Suppose $t^m_{P(K)} = 0$. Using \eqref{eqn:tg}, we see that \eqref{eqn:t2} is equivalent to 
\[
d + t_0 w +q w^2
= -d - t_0 w - qw^2 + w^2 s_{\overline{K}} +4 w^2 \nu(\overline{K}) - 2 w^2 + 2 d + 2 t_0 w + 2 q w^2.
\] 
\noindent This equation reduces to $2w^2 = w^2 s_{\overline{K}} + 4w^2\nu(\overline{K})$.  However, this equation has no solutions, since $\nu(\overline{K}) > 0$ and $s_{\overline{K}} \ge 0$.  
\end{itemize}
\end{proof}
\noindent Having proved the claim, it follows that if $\nu(K) < \nu(\overline{K})$, then $P(K)$ could not have been an L-space knot.  This completes the proof.

\end{proof}

\section{Proofs of Theorem~\ref{thm:jsj} and Proposition~\ref{prop:losatellite}}\label{sec:furtherresults}

Before proving Theorem~\ref{thm:jsj} and Proposition~\ref{prop:losatellite} we remind the reader of a standard fact about geometric structures and Dehn surgery which we will make use of repeatedly without reference (see \cite[Proposition 5]{Heil1974} and \cite[Section 5]{Thurston}).  Suppose that $M$ is a compact, orientable, irreducible manifold with incompressible torus boundary (e.g., the exterior of a non-trivial knot in $S^3$).  Then all but finitely many Dehn fillings of $M$ are irreducible and have the same numbers of hyperbolic and Seifert fibered pieces in their JSJ decompositions as $M$.

\subsection{JSJ decompositions and L-spaces}

\begin{proof}[Proof of Theorem~\ref{thm:jsj}]
In order to construct the family of manifolds described in the statement of the theorem, we will first construct an L-space satellite knot $K_{s,r}$ with $s$ Seifert fibered pieces and $r$ hyperbolic pieces in the JSJ decomposition.  The knot $K_{s,r}$ will be constructed by a sequence of satellite operations using cables and Berge-Gabai knots.  As discussed, all but finitely many surgeries on $K_{s,r}$ will then be irreducible rational homology spheres with the desired JSJ decomposition.  Since all surgeries with slope at least $2g(K_{s,r})-1$ will result in L-spaces (see Subsection \ref{sec:HFinput}), sufficiently large surgeries on $K_{s,r}$ will produce the desired infinite family.  

Recall that if $P$ is a torus knot standardly embedded in the solid torus, then the exterior of $P$ is Seifert fibered over the annulus with a single cone point.  
We first construct a knot $K_s$ as an $s$-fold iterated torus knot with appropriately chosen cabling parameters.  More specifically, we construct $K_s$ as follows.  If $s$ is 0, we simply take $K_s$ to be the unknot.  Otherwise, we begin with $K_1$, the positive $(m_1,n_1)$-torus knot, for some $m_1, n_1 \geq 2$.  Perform the $(m_2,n_2)$-cable, choosing $n_2/m_2 \geq 2g(K_1) - 1$, to obtain the knot $K_2$.  Inductively, we construct $K_i$ to be the $(m_i,n_i)$-cable of $K_{i-1}$, where we choose $n_i/m_i \geq 2g(K_{i-1}) - 1$.  The JSJ decomposition of the exterior of $K_s$ now consists of $s$ Seifert pieces.  Further, by \cite{Hedden2009}, $K_s$ is an L-space knot.  
 
Let $P_1$ be a positively twisted hyperbolic Berge-Gabai knot satisfying $\frac{b+ t_0 w + q w^2}{w^2} \geq 2g(K_s) - 1$.  We can construct $P_1$ as follows. Begin with any hyperbolic Berge-Gabai knot (i.e., hyperbolic in $S^1 \times D^2$; see \cite[Theorem 3.2 and p.17]{Berge1991} to obtain explicit examples).  Now add sufficiently many positive twists until the desired inequality is satisfied (fix $b, t_0,$ and $w$, and increase $q$) to obtain $P_1$.  As discussed in Section~\ref{sec2:subsec1}, adding positive twists preserves the property of being a Berge-Gabai knot; furthermore, this does not change the type of geometry on the knot exterior, and thus $P_1$ will still be hyperbolic.  If $s \neq 0$, we define $K_{s,1}$ as the satellite knot with companion $K_s$ and pattern $P_1$.  By Theorem~\ref{thm}, $K_{s,1}$ is an L-space knot.  If $s = 0$, take $K_{s,1}$ to be any hyperbolic L-space knot, such as the $(-2,3,7)$-pretzel knot \cite{FS1980}.  We now repeat this process $r$ times, i.e., to obtain $K_{s,i}$, satellite $K_{s,i-1}$ with pattern a hyperbolic Berge-Gabai knot satisfying $\frac{b+ t_0 w + q w^2}{w^2} \geq 2g(K_{s,i-1}) - 1$.  The process terminates at the knot $K_{s,r}$ whose exterior is irreducible and has $s$ Seifert and $r$ hyperbolic pieces in its JSJ decomposition.  We have that $K_{s,r}$ is an L-space knot by repeated application of Theorem~\ref{thm}.  As discussed above, this completes the proof.
\end{proof}

\subsection{Left-orderability}
Recall that a non-trivial group $G$ is {\em left-orderable} if there exists a left-invariant total order on $G$.  Examples of left-orderable groups include $\mathbb{Z}$ and $Homeo_+(\mathbb{R})$, while any group with torsion (e.g., a finite group) is not left-orderable.  It is natural to ask which three-manifold groups can be left-ordered.  Such groups are well-suited for this study due to the following theorem.  

\begin{thm}[Boyer-Rolfsen-Wiest \cite{BRW2005}]\label{thm:brw}
Let $Y$ be a compact, connected, irreducible, $P^2$-irreducible three-manifold.  If there exists a non-trivial homomorphism $f:\pi_1(Y) \to G$ where $G$ is left-orderable, then $\pi_1(Y)$ is left-orderable.  In particular, if there exists a non-zero degree map from $Y$ to $Y'$, where $\pi_1(Y')$ is left-orderable, then $\pi_1(Y)$ is left-orderable.  
\end{thm}  

Rather than define $P^2$-irreducible, we simply point out that if $Y$ is orientable, then irreducibility implies $P^2$-irreducibility.  For compact, orientable, irreducible three-manifolds with $b_1 > 0$, it then follows that their fundamental groups are always left-orderable.  However, there are more interesting phenomena for rational homology spheres; for example $+3/2$-surgery on the left-handed trefoil has left-orderable fundamental group, while $-3/2$-surgery has torsion-free, non-left-orderable fundamental group (this can be deduced for instance from \cite[Theorem 1.3]{BRW2005}).  Surprisingly, the left-orderability of the fundamental groups of three-manifolds is conjecturally characterized by Heegaard Floer homology. The following conjecture was made in \cite{BGW2012}:

\begin{customthm}{\ref{conj:lolspace}}[Boyer-Gordon-Watson]
Let $Y$ be an irreducible rational homology sphere.  Then $Y$ is an $L$-space if and only if $\pi_1(Y)$ is not left-orderable.  
\end{customthm}   

There exists a large amount of support for this conjecture, as it is known to be true for manifolds with Seifert or Sol geometry, branched double covers of non-split alternating links, graph manifold integer homology spheres, and many other families of examples (see for instance \cite{BB2013, BGW2012, P2009}).  We also remark that irreducibility is necessary, as $\Sigma(2,3,7) \# \Sigma(2,3,5)$ has non-left-orderable fundamental group, but is not an L-space.  

In the proof of Proposition~\ref{prop:losatellite} below, we remind the reader that we will be assuming Conjecture~\ref{conj:lolspace}.  

\begin{proof}[Proof of Proposition~\ref{prop:losatellite}]
Suppose that $P(K)$ is an L-space knot.  Then for all $\alpha\in \mathbb{Q}$ with $\alpha \geq 2g(P(K)) - 1$, we have $S^3_\alpha(P(K))$ is an L-space.  For all but finitely many such $\alpha$, we have that $S^3_\alpha(P(K))$ is irreducible as well.  Thus, by Conjecture~\ref{conj:lolspace}, we have that $\pi_1(S^3_\alpha(P(K)))$ is not left-orderable for $\alpha \gg 2g(P(K)) - 1$.  

We first study the pattern $P$.  By \cite[Proposition 13]{Clay2011}, for such $\alpha$, $\pi_1(S^3_\alpha(P))$ is not left-orderable.  Furthermore, for all but finitely many $\alpha$, we have that $S^3_\alpha(P)$ is irreducible.  Therefore, we appeal to Conjecture~\ref{conj:lolspace} to conclude that $P$ is an L-space knot.  

We modify the argument of \cite[Proposition 13]{Clay2011} to study the companion $K$.  Recall that $w$ represents the winding number of $P$ in the solid torus $V$.  We also consider the basis $(m,\ell)$ for $H_1(\partial V;\mathbb{Z})$ as given in Section~\ref{sec:mainsurgery}.  We choose $n \in \Z$ such that $\gcd(w,n) = 1$ and $n \gg 2g(P(K)) - 1$.  As discussed, we have $S^3_n(P(K))$ is irreducible and $\pi_1(S^3_{n}(P(K))$ is not left-orderable.  We consider the manifold $(P;n)$.  We have that the kernel of $i_*:H_1(\partial (P;n);\mathbb{Z}) \to H_1((P;n);\mathbb{Z})$ is generated by $nm + w^2\ell$ by Lemma~\ref{lem2:kernel}.  Since $\gcd(w,n) = 1$ by assumption, we have that the element $nm+w^2\ell$ is represented by a simple closed curve on $\partial (P;n)$ which bounds in $(P;n)$.  It then follows that there exists a degree one map $\phi:(P;n) \to S^1 \times D^2$, which restricts to a homeomorphism on the boundary (see for instance \cite[Lemma 2.2]{Rong1995}).  Since $nm+w^2\ell$ bounds in $(P;n)$, we must have that $\phi(nm+w^2\ell)$ is isotopic to $\{*\} \times D^2$.  

By extending $\phi$ to be the identity on the exterior of $K$, one obtains a degree one map from $S^3_n(P(K))$ to $S^3_{n/w^2}(K)$.  Since $S^3_n(P(K))$ is irreducible and $\pi_1(S^3_n(P(K)))$ is not left-orderable, we have that $\pi_1(S^3_{n/w^2}(K))$ is not left-orderable by Theorem~\ref{thm:brw}.  Since $w$ is fixed, by choosing sufficiently large $n$ with $\gcd(w,n) = 1$, we can arrange that $S^3_{n/w^2}(K)$ is irreducible as well.  Again, by Conjecture~\ref{conj:lolspace}, $K$ is an L-space knot.              
\end{proof}     

\bibliographystyle{amsalpha2}

\bibliography{Reference}
\end{document}